\documentclass[12pt]{article}
\usepackage{amsmath,amsthm,amsfonts}
\usepackage{times}
\usepackage{url}

\newtheorem{thm}{Theorem}[section]

\newtheorem{lem}[thm]{Lemma}
\newtheorem{ex}[thm]{Example}
\newtheorem{cor}[thm]{Corollary}

\theoremstyle{remark}

\DeclareMathOperator{\rank}{\mathrm{rank}}
\DeclareMathOperator{\aut}{\mathrm{Aut}}

\newcommand{\cB}{\mathcal{B}}
\newcommand{\cD}{\mathcal{D}}
\newcommand{\cS}{\mathcal{S}}
\newcommand{\cF}{\mathcal{F}}

\newcommand{\GF}{{\mathop{\rm GF}}}

\newcommand{\fp}{\GF(p)}
\newcommand{\PG}{\mathrm{PG}}
\newcommand{\AG}{\mathrm{AG}}

\begin{document}

\title{
Counting Steiner triple systems with classical parameters and prescribed rank
}

\author{
Dieter Jungnickel\thanks{Mathematical Institute,
University of Augsburg, D-86135 Augsburg, Germany.} \hspace{1mm}  and
  Vladimir D. Tonchev\thanks{
Department of Mathematical Sciences,
Michigan Technological University
Houghton, MI 49931, USA.}
}

\maketitle

\begin{abstract}
By a famous result of Doyen, Hubaut and Vandensavel \cite{DHV}, the 2-rank of a Steiner 
triple system on $2^n-1$ points is at least $2^n -1 -n$,
and equality holds only for the classical point-line design in the projective 
geometry $PG(n-1,2)$. It follows from results of Assmus \cite{A}
that, given any integer $t$ with $1 \leq t \leq n-1$, there is  
a code $C_{n,t}$ containing representatives of all isomorphism classes
of STS$(2^n-1)$ with 2-rank at most $2^n -1 -n + t$. Using a mixture of coding 
theoretic, geometric, design theoretic and combinatorial
arguments, we prove a general formula for the number of distinct STS$(2^n-1)$ 
with 2-rank at most $2^n -1 -n + t$ contained in this code. This generalizes 
the only previously known cases, $t=1$, proved by Tonchev \cite{T01}
in 2001, $t=2$, proved by V. Zinoviev and D. Zinoviev \cite{ZZ12} in 2012,
 and $t=3$ (V. Zinoviev and D. Zinoviev \cite{ZZ13}, \cite{ZZ13a} (2013),
D. Zinoviev \cite{Z16} (2016)), 
while also unifying and simplifying the proofs. 

This enumeration result  allows us to prove  lower and upper
bounds for the number of isomorphism classes of STS$(2^n-1)$ with 
2-rank exactly (or at most) $2^n -1 -n + t$.
Finally, using our recent systematic study of the ternary block codes
 of Steiner triple systems \cite{JT}, we obtain analogous
results for the ternary case, that is, 
for STS$(3^n)$ with 3-rank at 
most (or exactly) $3^n -1 -n + t$.

We note that this work provides the first two infinite families of 2-designs 
for which one has non-trivial
 lower and upper bounds for the number of non-isomorphic examples 
with a prescribed $p$-rank in almost the entire range of possible ranks.

\end{abstract}


\vspace{2mm}

\section{Introduction}

We assume familiarity with basic facts and notation concerning combinatorial designs \cite{BJL} and codes \cite{AK}, \cite{HP}.
Throughout this paper,  an incidence matrix of a design will have its rows indexed by the blocks, while the columns are indexed by the points of the corresponding design.

It was shown by Doyen, Hubaut and Vandensavel \cite{DHV} that only the binary and ternary codes of Steiner triple systems can be interesting: for primes $p \neq 2,3$,  the $\fp$-code of any STS($v$)  has full rank $v$.
The classical examples of STS are provided by the point-line designs in binary projective and ternary affine spaces.
By a famous result of Doyen, Hubaut and Vandensavel, the 2-rank of a Steiner triple system on $2^n-1$ points is at least $2^n -1 -n$,
and equality holds only for the classical point-line design in the projective geometry $PG(n-1,2)$. 
An analogous result also holds for the ternary case, that is, for STS$(3^n)$.

In \cite{A}, Assmus proved that the incidence matrices
 of all Steiner triple systems on $v$ points which have
 the same $2$-rank generate equivalent binary codes,
and gave an explicit description of a generator matrix for such a code.
 In our recent systematic study of the binary and ternary block codes
 of Steiner triple systems \cite{JT}, 
we also obtained a corresponding result for the ternary case. 
In all these cases, 
we give an explicit parity check matrix for the code in question.

Using these results, we will deal with the enumeration problem for
 STS on $2^n-1$ or $3^n$ points with a prescribed 2-rank or 3-rank, 
respectively.
In Section \ref{secbin}, we will use a mixture of coding theoretic,
geometric, design theoretic and combinatorial arguments to prove 
a general formula for the number of 
\emph{distinct} STS$(2^n-1)$ with 2-rank at most $2^n -1 -n + t$ contained 
in the relevant code.
Our approach
differs from the one used by the second author in \cite{T01} 
to find an explicit formula for the $STS(2^n -1)$ of 2-rank $2^n -n$, and
 is somewhat reminiscent of the constructions of STS$(2^n-1)$ with small 2-rank 
given by Zinoviev and Zinoviev \cite{ZZ12,ZZ13}, 
who also briefly mention a possible extension to higher ranks in \cite{ZZ13}.
However, our treatment will rely essentially on design theoretic and geometric methods,
 whereas \cite{T01}, \cite{ZZ12,ZZ13} 
use almost exclusively the language of coding theory.
This allows us to give a unified, considerably shorter 
and, in our opinion, more transparent presentation.

The ternary case has not been studied before, except for our 
recent (mainly computational) work on STS(27) with 3-rank 24 \cite{JMTW}.
In Section \ref{sectern} -- which is completely parallel 
to Section \ref{secbin} -- we provide general enumeration results 
also for the ternary case. 
Namely, we derive a formula for the exact number of distinct
$STS(3^n)$ having 3-rank at most $3^n - n - 1 +t$ that are contained
in a ternary $[3^n, 3^n - n - 1 +t]$ code having
a parity check matrix obtained
by deleting $t$ rows from the generator matrix of the first order
Reed-Muller code of length $3^n$.

Finally, in Section \ref{secnoniso}, we use our enumeration of the 
distinct examples in the relevant code to obtain both lower and upper
bounds for the number of \emph{isomorphism classes} of STS$(2^n-1)$ having 
2-rank exactly (or at most) $2^n -1 -n + t$,
as well as lower and upper
bounds for the number of \emph{isomorphism classes} of STS$(3^n)$
having 3-rank exactly (or at most) $3^n - 1 - n +t$.

 The lower bounds appear to be quite strong and show 
the expected combinatorial explosion even for STS with small rank.
As examples, we show that the number of isomorphism classes of STS(31) 
with 2-rank at most 29 is larger than $10^{24}$;
similarly, the number of isomorphism classes of STS(27) 
with 3-rank at most 25 is larger than $10^{19}$.

To the best of our knowledge, the results of this paper 
not only generalize the previously known special cases
in the binary case ($t=1$: Tonchev \cite{T01}, $t=2$:
V. Zinoviev and D. Zinoviev \cite{ZZ12}, $t=3$: D. Zinoviev \cite{Z16}),
and develop an analogous general theory for the ternary case,
but also provide the first two infinite families of 2-designs for which one has non-trivial
 lower and upper bounds for the number of non-isomorphic examples with a prescribed $p$-rank in almost the entire range of possible ranks.
(The only cases where our bounds do not apply are for designs having full 2-rank $v$, or 3-rank $v-1$.)

\section{The binary case}\label{secbin}

The following result was recently proved in \cite[Theorem 4.1]{JT}.
It generalizes a property of $STS(2^n -1)$ having 2-rank at most $2^n - n+2$
proved in \cite{ZZ13}.

\begin{thm}\label{binary}
Let $D$ be a Steiner triple system on $v$ points, and assume that $D$ has $2$-rank $v-m$, where $m \geq 1$.\\
(i) The binary linear code $C$ of length $v$ and dimension $v-m$ spanned by the
incidence matrix $A$ of $D$ has an $m \times v$ parity check matrix $H$ whose column set consists 
of $w$ copies of the column set of the $m \times (2^m -1)$ parity check matrix $H_m$ of the binary Hamming code
of length $2^m -1$,  and  $w-1$ all-$0$ columns (for some $w\,\geq\, 1$). In particular, $v$ has the form $v= w \cdot 2^m -1$.\\[1mm]
(ii) The dual code $C^{\perp}$ is an equidistant code for which all nonzero codewords have weight $d=(v+1)/2$.  \qed
\end{thm}

An immediate consequence is the following result first proved by Assmus \cite[Theorem 4.2]{A}:

\begin{cor}\label{corbinarycode}
The binary linear code spanned by the incidence vectors of the blocks of a Steiner triple system on $v$ points with $2$-rank $v-m$
contains representatives of all isomorphism classes of Steiner triple systems on $v$ points having $2$-rank  $v-m$.  \qed
\end{cor}

The classical examples of STS with a non-trivial binary code are the point-line designs in binary projective spaces. 
Here one has the following result \cite[Theorem 4.5]{JT} originally established by Doyen, Hubaut and Vandensavel \cite{DHV} (via a geometric approach), but without the statement on the parity check matrix:

\begin{thm}\label{binaryHamada}
Let $C$ be the binary linear code spanned by the incidence matrix $A$  of a Steiner triple system $D$ on $2^n-1$ points. Then
\begin{equation}\label{eqb2}
\dim C \,=\, \rank_{2}A \,\geq\, 2^n -1 -n.
\end{equation}
Equality holds in (\ref{eqb2}) if and only if the $n \times (2^n -1)$ parity check matrix of $C$ is a parity check matrix of the Hamming code of length $2^n -1$
(equivalently, a generator matrix for the simplex code of this length), 
in which case $D$ is isomorphic to the design $\PG_1(n-1,2)$ of points and lines in $\PG(n-1,2)$. \qed
\end{thm}

We will also need the following strengthening of Corollary \ref{corbinarycode} \cite[Theorem 4.6]{JT}:

\begin{thm}\label{allSTSbinary}
The binary linear code $C$ spanned by the incidence vectors of the blocks of a Steiner triple system on $v=2^n-1$ points with $2$-rank $2^n -1-k$, where $k \le n$, 
contains representatives of all isomorphism classes of Steiner triple systems on $2^n-1$ points having $2$-rank at most $2^n-1-k$.

In particular, the binary code of the classical system $\PG_1(n-1,2)$  is a subcode of  the code of every other STS on $2^n-1$ points. \qed
\end{thm}

We now fix some notation in order to study the Steiner triple systems on $2^n-1$ points with prescribed 2-rank $2^n-1-n+t$, where $1 \leq t \leq n-1$. 
By Theorem \ref{binaryHamada}, the binary $[2^n-1,2^n -n-1]$ code spanned by the incidence matrix of the classical $STS(2^n-1)$ has a parity check matrix  $H_n$ 
whose column set  consists  of all distinct non-zero vectors in $GF(2)^n$.

Moreover,  for $1\le t \le n-1$, the $(n-t) \times (2^n-1)$  matrix $H_{n,t}$ obtained by deleting (arbitrarily chosen) $t$ rows of $H_n$
is the parity check matrix  of a binary $[2^n, 2^n -1-n+t]$ code which contains representatives of all isomorphism classes of $STS(2^n-1)$ having 2-rank at most $2^n -1-n+t$. 
We note that the column set of $H_{n,t}$ consists  of all vectors of $GF(2)^{n-t}$, where each non-zero vector appears exactly $2^t$ times, whereas the all-zero vector $\mathbf{0}$ appears exactly $2^t-1$ times. 
These facts are easy consequences of the results of \cite{JT} stated above.
A similar result for the special case $t=3$ is given in \cite{Z16,ZZ13}.

Note that the matrices $H_{n,t}$ are all  unique up to a permutation of their columns. To fix the notation completely, we will henceforth assume that the columns are ordered lexicographically.
Now fix $n$ and $t$, and let $C=C_{n,t}$ be the binary code with parity check matrix $H=H_{n,t}$. In what follows, we will use the abbreviations $N = 2^n-1$, $T = 2^t-1$, and $M= 2^{n-t}-1$. 

\vspace{1mm}
For later use, we first describe the automorphism group of $C$, a result due to Assmus \cite[Corollary 3.7]{A}:

\begin{thm}\label{autgpbinary}
The code $C$ is invariant under a group $G$ of order
$$ T! \cdot \big((T+1)!\big)^M \cdot |PGL(n-t,2)|.$$
The group $G$ is a wreath product of two groups $G_1$ and $G_2$.
Here $G_1$ is the direct product of the symmetric group $S_T$ with $M$ copies of the symmetric group $S_{T+1}$, 
where $S_T$ acts on the set of all-zero columns of the parity check matrix $H$ and where each copy of $S_{T+1}$ acts on a set of identical non-zero columns of $H$,
and  $G_2$ is the collineation group $PGL(n-t,2)$ of the $(n-1-t)$-dimensional projective geometry $\Pi := PG(n-1-t,2)$.
\end{thm}

\begin{proof}
Note that the non-zero columns of $H$ are vectors representing the points of $PG(n-1-t,2)$, with each point repeated $T$ times.
Then the statements of the theorem follow directly from the structure of the column set of the parity check matrix $H$, as described above.
\end{proof}

Since the block set of any $STS(N)$ having 2-rank at most $N-n+t$ consists (up to isomorphism) of the supports of a suitable set of words of weight 3 in $C$,
we begin by studying the triple system $\cD$ on the point set $V = \{1,\ldots,N\}$ of column indices which has the supports of \emph{all} words of weight 3 in $C$ as blocks.
Whenever convenient, we will  identify the points of $\cD$  with the columns of $H$.
Similarly, we will usually not distinguish between a block of $\cD$ and the corresponding word of weight 3 in $C$. 

We split $V$ according to the structure of $H$  as follows, taking into account the lexicographical ordering of the columns:
\begin{itemize}
\item  Let $V_0 = \{1,\ldots,T\}$ be the set of all-zero columns of $H$.
\item  The remaining points are split into \emph{groups} $G_1,\ldots,G_M$ of $T+1$ identical columns each. Thus $G_i = \{T+(i-1)(T+1)+1,\ldots,T+i(T+1)\}$.
\end{itemize}
Note that the groups correspond to the $M$ points of the $(n-1-t)$-dimensional projective geometry $\Pi = PG(n-1-t,2)$.

A {\it group divisible design} GDD$(m,n,k,\lambda_1, \lambda_2)$
(or GDD for short), is an incidence structure
with $mn$ points and blocks of size $k$, such that the points are partitioned into $m$
groups of size $n$, and every two points that belong to the same group
appear together in $\lambda_1$ blocks, while every two points from
different groups appear together in $\lambda_2$ blocks.

\begin{lem}\label{blocksbinary}
Let $x$ and $y$ be two distinct points of $\cD$, and let $B=\{x,y,z\}$ be any block containing these two points. Then one of the following cases occurs:
\begin{enumerate}
\item  If $x,y \in V_0$, then also $z \in V_0$.
\item  If one of the two points, say $x$, belongs to $V_0$ and the other point $y$ belongs to some group $G_i$, then also $z \in G_i$.
\item  If $x$ and $y$ belong to different groups $G_i$ and $G_j$, then $z$ belongs to a group $G_k$ with $k \neq i,j$.
    	Moreover, if $\bar x$ and $\bar y$ are the  (distinct) points of $\Pi$ corresponding to $x$ and $y$, then $\bar z$ is the third point  $\bar x+ \bar y$ on the line of $\Pi$ through $\bar x$ and $\bar y$.
    	In other words, $G_k$ is the group defined by the point $\bar x+ \bar y$ of $\Pi$.
\end{enumerate}
In particular, any block $B$ of $\cD$ joining two points in distinct groups induces a line of $\Pi$.
\end{lem}

\begin{proof}
First assume $x,y \in V_0$ and note that any column $z \notin V_0$ contains an entry 1 in some row of $H$. 
Since $B$ has to be orthogonal to all rows of $H$, the assumption $z \notin V_0$ would lead to a contradiction, which proves (i).

\vspace{1mm}
Next, let $x \in V_0$ and $y \notin V_0$, so that also $z \notin V_0$. 
Suppose that the columns $y$ and $z$ are distinct vectors in $GF(2)^{n-t}$. Then there is at least one row of $H$ where one of these two columns has entry 0 and the other has entry 1.
Such a row would not be orthogonal to $B$, and hence $y$ and $z$ have to be identical vectors, establishing (ii).

\vspace{1mm}
A similar argument as for case (i) shows that the case $x,y \in G_i$ and $z \notin V_0$ cannot occur.
Thus we are left with the case where $x$ and $y$ belong to different groups $G_i$ and $G_j$. 
Let $\bar x$ and $\bar y$ be the  (distinct) points of $\Pi$ corresponding to $x$ and $y$, 
and denote the group determined by the point $\bar x+ \bar y$ of $\Pi$ by $G_k$.
It is clear that $B' = \{x,y,w\}$ is a block, whenever $w \in G_k$.

Now suppose $z \in G_h$, where $h \neq k$. Then the sum of the code words $B$ and $B'$ is a word $c$ of weight 2 in $C$,
and the two non-zero entries of $c$ belong to the columns $w$ and $z$ in the distinct groups $G_k$ and $G_h$.
Consider  the matrix obtained by using just one of the $T+1$ columns in each group, 
that is, the parity check matrix $H' := H_{n-t}$ for the code $C'$ determined by the lines in the  projective geometry $\Pi = PG(n-1-t,2)$.
As $c$ is orthogonal to all rows of $H$, the vector $c'$ of length $2^{n-t}-1$ and weight 2
with entries 1 in positions $k$ and $h$ has to be orthogonal to all rows of $H'$.
This contradicts the well-known fact that the minimum weight vectors of $C'$ have weight 3 (they are the incidence vectors of  the lines of $\Pi$).
This shows $z \in G_h$ and proves (iii).
 \end{proof}

The following theorem is a simple consequence of Lemma \ref{blocksbinary}:

\begin{thm}\label{GDDbinary}
The set of blocks of $\cD$ splits as follows:
\begin{itemize}
\item   The blocks contained in $V_0$ form a complete $2$-$(T,3,T-2)$ design.
\item   The blocks disjoint from $V_0$ give a GDD with the $M$ groups  $G_1,\ldots,G_M$ of size $T+1$ each,
	 where two points in the same group are not joined at all, whereas  two points in different groups are joined by  $T+1$ blocks.
\item  All other blocks contain two points in the same group and intersect $V_0$ in a unique point. 
	Any two points in the same group are in exactly $T$ blocks of this type.
\end{itemize}
\end{thm}

\begin{proof}
This follows easily from Lemma \ref{blocksbinary}: it suffices to observe that any choice of three points $x,y,z$ satisfying the
conditions in one of the three cases of the lemma gives a block.
\end{proof}

As a further consequence of Lemma \ref{blocksbinary}, we can also describe the structure of any STS$(N)$ contained in $C$
(that is, of any subset of $N(N-1)/6$ blocks of $\cD$ forming an STS).
While this description bears some resemblance to Theorem 4.1 of  Assmus \cite{A}, the use of the GDD $\cD$ allows a considerably more transparent result, which is suitable for counting purposes.

\begin{thm}\label{STSbinary}
Let $\cS$ be an arbitrary Steiner triple system STS($N$) contained in $\cD$. Then the block set $\cB$ of $\cS$ splits as follows:
\begin{itemize}
\item   a set $\cB_0$ of\, $T(T-1)/6$ blocks contained in $V_0$, such that $(V_0,\cB_0)$ is a Steiner triple system $\cS_0$ on the $T$ points in $V_0$;
\item   for all $i=1,\ldots,M$, a set $\cB_i$ of $T(T+1)/2$ blocks joining a point $x$ of $V_0$ to two points $y$, $y'$ in the group $G_i$. 
           For each choice of $x$, there are $(T+1)/2 = 2^{t-1}$ such blocks, and the sets $\{y,y'\}$ occurring in these $2^{t-1}$ blocks  yield a $1$-factor $F_x^{(i)}$ of the complete graph on $G_i$;
           moreover, the $1$-factors $F_x^{(i)}$ ($x \in V_0$) form a $1$-factorization  $\cF^{(i)}$ of this complete graph. 
\item   for each line $\ell$ of the projective geometry $\Pi = PG(n-1-t,2)$, a set $\cB_\ell$ of $(T+1)^2 = 2^{2t}$ blocks forming a transversal design TD$[3;T+1]$ on the three groups determined by the points of $\ell$.
\end{itemize}
\end{thm}

\begin{proof}
By part (i) of Lemma \ref{blocksbinary}, we necessarily obtain a sub-STS($T$) of $\cS$ on the point set $V_0$.
(In the terminology of Assmus \cite{A}, $\cS_0$ is the {\it trivializing subsystem} of $\cS$.)

Part (ii) of Lemma \ref{blocksbinary} shows that a point $x \in V_0$ has to be joined to a point $y \in G_i$ by a block of the form $\{x,y,y'\}$ with $y' \in G_i$. 
Since every point $y \in G_i$ is joined to $x$ by exactly one block in $\cB$, the sets $\{y,y'\}$ occurring in such a block have to form a $1$-factor $F_x^{(i)}$ of the complete graph on $G_i$.
As any two points in $G_i$ also determine a unique block in $\cB$, no pair $\{y,y'\}$ can occur in two of these 1-factors, so that the $T$ 1-factors $F_x^{(i)}$ indeed give a 1-factorization.

By part (iii) of Lemma \ref{blocksbinary}, the remaining blocks of $\cS$ have to consist of points in three distinct groups which induce a line of $\Pi$.
Now let $\ell = \{\bar x, \bar y, \bar z\}$ be such a line, and let the three groups determined by the points of $\ell$ be $G_i$, $G_j$ and $G_k$.
As any two points in different groups have to be on a unique block in $\cB$, the set of blocks of $\cS$ inducing the line $\ell$ obviously has to form a TD on these three groups. 
\end{proof}

The preceding results lead to a generic formula for the number of distinct Steiner triple systems contained in $\cD$.
For this, we shall denote
\begin{itemize}
\item  the number of distinct Steiner triple systems on a $v$-set by $N_1(v)$;
\item  the number of distinct 1-factorizations of the complete graph on $2k$ vertices by $N_2(2k)$;
\item  and the number of distinct transversal designs $TD[3;g]$ on three specified groups of size $g$ by $N_3(g)$.
\end{itemize}
The desired formula will be obtained as a consequence of the  splitting of the block set of $\cD$ given in Theorem \ref{GDDbinary} and the structure of any STS contained in $\cD$ described in Theorem \ref{STSbinary}:

\begin{thm}\label{nrSTSbinary}
The number $s(n,t)$ of distinct Steiner triple systems (with $2$-rank at most $N-n+t$) contained in the triple system $\cD$ 
formed by the supports of the words of weight $3$ in the binary code $C$ with parity check matrix $H_{n,t}$, where $1 \leq t \leq n-1$, is given by
\begin{equation}\label{eqnrbin}
s(n,t) \,=\, N_1(T) \cdot \big(N_2(T+1)\cdot T! \big)^M \cdot  N_3(T+1)^{M(M-1)/6},
\end{equation}
where  $N = 2^n-1$, $T = 2^t-1$, and $M = 2^{n-t}-1$. 
\end{thm}

\begin{proof}
First, we have to select an STS($T)$ on $V_0$ contained in $\cD$; as the blocks of $\cD$ contained in $V_0$ form the complete design on $V_0$, we can do this in $N_1(T)$  ways.

Then we need to join the points in any given group $G_i$ to the points in $V_0$ by using a 1-factorization $\cF^{(i)}$ of the complete graph on $G_i$, as described in Theorem \ref{STSbinary}.
As noted in the proof of Theorem \ref{GDDbinary}, all triples  of the required form $\{x,y,y'\}$ are indeed blocks of $\cD$, so that we can choose $\cF^{(i)}$  arbitrarily from the $N_2(T+1)$ possible 1-factorizations.
In addition, we have to decide how the $T$  points in $V_0$  are matched to the 1-factors in $\cF^{(i)}$, which can be done in $T!$ ways, for each choice of $\cF^{(i)}$.
This process has to be done for all $M$ groups, which leads to the second factor in the formula \eqref{eqnrbin}.

Finally, given any line $\ell$ of $\Pi$, we have to select a transversal design on the three groups determined by $\ell$.
Again, all triples consisting of one point in each of these groups are blocks of $\cD$, so that the required TD can be chosen in $N_3(T+1)$ ways.
This  has to be done for all $M(M-1)/6$ lines of $\Pi$, which results in the last factor in formula \eqref{eqnrbin}.
\end{proof}

It was pointed out to us by
an anonymous reviewer that our arguments, statements and expressions
in Lemma \ref{blocksbinary} 
and Theorems \ref{GDDbinary} and \ref{STSbinary} are similar to those of Lemma 8
and Theorem 3 from \cite{ZZ13}. However, our Lemma \ref{blocksbinary},
and our Theorems \ref{GDDbinary}, \ref{STSbinary} and \ref{nrSTSbinary} 
apply to arbitrary 2-rank $k \le 2^n -1$
while Lemma 8 and Theorem 3 from  \cite{ZZ13} consider only  2-rank
$k \le 2^ n -  n +2$.
Thus, our results generalize the results from \cite{ZZ13}.

We now apply Theorem \ref{nrSTSbinary} to give (unified and considerably shorter) proofs for the cases which have been studied before, namely $t \in \{1,2,3\}$.
Here the required data are either easy to check directly (for very small parameters) or at least known (by computer searches).

For this, we  note that $N_3(g)$ agrees with the number of Latin squares (or labeled quasigroups) of order $g$,
which is the special case $k=3$ of the well-known correspondence between transversal designs $TD[k;g]$, orthogonal arrays $OA(k,g)$ and sets of $k-2$ mutually orthogonal Latin squares of order $g$;
see, for instance, \cite[Lemma VIII.4.6]{BJL}.

The special case $t=1$ is particularly simple and gives the following result due to Tonchev \cite{T01}:

\begin{cor}\label{nrSTSbinary+1}
The number of distinct Steiner triple systems contained in the binary code with parity check matrix $H_{n,1}$ is given by
\begin{equation}\label{eqnrbin1}
s(n,1) \,=\, 2^{(2^{n-1}-1)(2^{n-2}-1)/3}
\end{equation}
\end{cor}

\begin{proof}
Here $T=1$ and $M=2^{n-1}-1$. Trivially, $N_1(1) = N_2(2) = 1$. Finally, one easily checks $N_3(2) = 2$.  
\end{proof}

Similarly, the cases $t=2$ and $t=3$ yield the following results first established by Zinoviev and Zinoviev \cite{ZZ12,ZZ13}:

\begin{cor}\label{nrSTSbinary+2}
The number of distinct Steiner triple systems contained in the binary code  with parity check matrix $H_{n,2}$ is given by
\begin{equation}\label{eqnrbin2}
s(n,2) \,=\, 6^{2^{n-2}-1} \cdot  576^{(2^{n-2}-1)(2^{n-3}-1)/3}
\end{equation}
\end{cor}

\begin{proof}
Here $T=3$ and $M=2^{n-2}-1$. Trivially, $N_1(3) = N_2(4) = 1$. Finally, one can show $N_3(4) = 24^2 = 576$;
this could still be checked directly, but is, of course, known: see \cite[A002860]{OE}.
\end{proof}

\begin{cor}\label{nrSTSbinary+3}
The number of distinct Steiner triple systems contained in the binary code  with parity check matrix $H_{n,3}$ is given by
\begin{equation}\label{eqnrbin3}
s(n,3) \,=\, 30 \cdot 31449600^{2^{n-3}-1} \cdot  108776032459082956800^{(2^{n-3}-1)(2^{n-4}-1)/3}
\end{equation}
\end{cor}

\begin{proof}
Here $T=7$ and $M=2^{n-3}-1$. 
It is easy to see that $N_1(7) = 30$: up to isomorphism, the projective plane $\Pi_0=PG(2,2)$ is the only STS(7).
Also, the automorphism group of $\Pi_0$ is the group $PGL(3,2)$ of order 168, so that indeed $N_1(7) = 7!/168 = 30$.

The values for $N_2(8)$ and $N_3(8)$  can be found in the On-Line Encyclopedia of Integer Sequences:
one has $N_2(8) = 6240$, see \cite[A000438]{OE};
and  $N_3(8) = 108776032459082956800$, see \cite[A002860]{OE}.
\end{proof}

The next case, $t=4$, cannot be evaluated explicitly at present. While it would still be possible to compute $N_1(15)$
(as the STS(15) and their automorphism groups are classified), the values $N_2(16)$ and $N_3(16)$ are not known.
According to \cite[A000438]{OE}, $N_2(16)$ is approximately $1.48 \cdot 10^{44}$.
Clearly, an explicit evaluation for this case would not lead to a particularly illuminating formula, anyway.

Nevertheless, we feel that the generic formula \eqref{eqnrbin} together with the three smallest examples just discussed
provides a lot of insight into this particular enumeration problem. Certainly, it illustrates the combinatorial explosion
of the number of distinct STS$(2^n-1)$ even when we prescribe a small 2-rank.

Consequences for the number of isomorphism classes of such STS will be discussed in Section \ref{secnoniso}.

\section{The ternary case}\label{sectern} 

We now turn our attention to the ternary case. Using the results of \cite{JT} for the ternary situation, 
this can be done with the same approach as in Section \ref{secbin}; everything works in complete analogy.
Actually, the ternary case is somewhat simpler, which is due to the fact that there are no all-zero columns in the parity check matrix.
This means that we will have only two cases in the ternary analogue of Lemma \ref{blocksbinary},
and that we will not encounter any special substructure comparable to the trivializing subsystem on the set $V_0$ in Theorem \ref{STSbinary}.

We first recall the necessary material for the ternary situation from our recent paper \cite{JT}.
In particular, we need the following result \cite[Theorem 5.1]{JT}:

\begin{thm}\label{ternary}
Let $D$ be a Steiner triple system on $v$ points, and assume that $D$
 has $3$-rank $v-m$, where $m \geq 2$. 
Then $v$ is of the form $v = 3^{m-1} \cdot w$, where $w \equiv 1 \mbox{ or } 3 \bmod 6$, 
and the ternary linear code $C$ of length $v$ and
 dimension $v-m$ spanned by the incidence matrix $A$ of $D$ 
has an $m \times v$ parity check matrix $H$ with first row the all-$1$ vector $\mathbf{j}$, 
while the remaining positions in the columns of $H$ contain each vector
 in $\GF(3)^{m-1}$ exactly $w$ times.

Moreover, the dual code $C^\perp$ consists of the scalar multiples of the all-$1$ vector $\mathbf{j}$ and  codewords of constant weight $2v/3$, 
with half of the non-zero entries equal to $1$ and the other half equal to $2$.  \qed
\end{thm}

\begin{cor}\label{corternarycode}
The ternary linear code $C$ spanned by the incidence vectors of the blocks of a Steiner triple system on $v$ points with $3$-rank $v-m$, where $m \geq 2$, 
contains representatives of all isomorphism classes of Steiner triple systems on $v$ points having $3$-rank $v-m$.  \qed
\end{cor}

The classical examples of STS with a non-trivial ternary code are the point-line designs in ternary affine spaces. 
As in the binary case, one has a result \cite[Theorem 5.6]{JT} originally established by Doyen, Hubaut and Vandensavel \cite{DHV} 
(via a geometric approach), but without the statement on the parity check matrix:

\begin{thm}\label{ternaryHamada}
Let $C$ be the ternary linear code spanned by the incidence matrix $A$  of a Steiner triple system $D$ on $3^n$ points. Then
$$\dim C \,=\, \rank_{3}A \,\geq\, 3^n -1 -n.$$
Equality holds if and only the $(n+1) \times 3^n$ parity check matrix of $C$ is a generator matrix for the ternary first order Reed-Muller code of length $3^n$,
in which case $D$ is isomorphic to the design $\AG_1(n,3)$ of points and lines in $\AG(n,3)$.  \qed
\end{thm}

We also require the following strengthening of Corollary \ref{corternarycode} \cite[Theorem 5.8]{JT}:

\begin{thm}\label{allSTSternary}
The ternary linear code $C$ spanned by the incidence vectors of the blocks of a Steiner triple system on $v=3^n$ points with $3$-rank $3^n -k$, 
where $k\le n$, contains representatives of all isomorphism classes of Steiner  triple systems on $3^n$ points having $3$-rank smaller than or equal to $3^n-k$.

In particular, the ternary code of the classical system  $\AG_1(n,3)$ is a subcode of  the code of every other STS on $3^n$ points. \qed
\end{thm}

We can now set up the notation needed to study the Steiner triple systems on $3^n$ points with prescribed 3-rank $3^n-n-1+t$, where $1 \leq t \leq n-1$. 
By Theorems \ref{ternary} and \ref{ternaryHamada}, the ternary $[3^n,3^n -n-1]$ code spanned by the incidence matrix of the classical $STS(3^n)$ has a parity check matrix  $H_n$ of the following form:
\begin{equation}\label{hn}
H_n \,=\, \begin{pmatrix}
                         1 & \ldots & 1 \\
			& B_n &
\end{pmatrix},
\end{equation}
where $B_n$ is an $n\times 3^n$ matrix whose column set  consists of all distinct vectors in $GF(3)^n$.
(Note that $H_n$ is the usual generator matrix for the ternary first order Reed-Muller code of length $3^n$.)
Furthermore, if $1\le t \le n-1$, and $B_{n,t}$ is an $(n-t)\times 3^n$ matrix obtained by deleting $t$ rows of $B_n$, then the $(n+1-t)\times 3^n$ matrix $H_{n,t}$ given by
\begin{equation}\label{hi}
H_{n,t} \,=\, \begin{pmatrix}
                         1 & \ldots & 1 \\
			& B_{n,t}  &
\end{pmatrix}
\end{equation}
is the parity check matrix  of a ternary $[3^n, 3^n -n -1 +t]$ code which contains representatives of all isomorphism classes of $STS(3^n)$ having 3-rank at most $3^n -n -1 +t$.
We note that the column set of the $(n-t)\times 3^n$ matrix $B_{n,t}$ in Equation (\ref{hi}) consists  of all vectors of $GF(3)^{n-t}$, 
where each vector appears exactly $3^t$ times as a column of $B_{n,t}$. These facts are easy consequences of the results of \cite{JT} stated above.

Note that the matrices $B_{n,t}$ and  $H_{n,t}$ are all unique up to a permutation of their columns. 
To fix the notation completely, we will again assume that the columns are ordered lexicographically.
Now fix $n$ and $t$, and let $C=C_{n,t}$ be the ternary code with parity check matrix $H=H_{n,t}$. In what follows, we will use the abbreviations  $T = 3^t$ and $M= 3^{n-t}$. 

\vspace{1mm}
As in the binary case, we first describe the automorphism group of $C$.

\begin{thm}\label{autgpternary}
The code $C$ is invariant under a group $G$ of order
$$\big(T!\big)^M \cdot |AGL(n-t,3)|.$$
The group $G$ is a wreath product of two groups $G_1$ and $G_2$.
Here $G_1$ is the direct product of $M$ copies of the symmetric group $S_T$, where each copy acts on a set of identical columns of $H$,
and  $G_2$ is the collineation group $AGL(n-t,3)$ of the $(n-t)$-dimensional affine geometry $AG(n-t,3)$. \qed
\end{thm}

\begin{proof}
Note that the columns of the $(n-t)\times 3^n$ sub-matrix $B_{n,t}$ of the parity check matrix $H$ has as columns the vectors representing the points of $AG(n-t,3)$, each point repeated $T$ times.
Then the statements of the theorem follow directly from the structure of $H$ described above.
\end{proof}

Similar to the binary case, the block set of any $STS(3^n)$ having 3-rank at most $3^n-n-1+t$ consists (up to isomorphism) of the supports of a suitable set of words of weight 3 in $C$.
Let $x=(x_1,\ldots,x_n)$ be a codeword of weight 3 with nonzero components $x_i, x_j, x_k$.
Since $x$ is orthogonal to the all-one vector (the first row of $H_{n,t}$), we have $x_i = x_j = x_k$.
Without loss of generality, we may assume $x_i = x_j = x_k = 1$.

We now study the triple system $\cD$ on the point set $V = \{1,\ldots,N\}$ of column indices which has the supports of all words of weight 3 in $C$ (with non-zero entries 1) as blocks.
Whenever convenient, we will again identify the points of $\cD$  with the columns of $H$.
Similarly, we will usually not distinguish between a block of $\cD$ and the corresponding word of weight 3 with non-zero entries 1 in $C$. 

We split $V$ according to the structure of $H$ (taking into account the lexicographical ordering of the columns) into \emph{groups} $G_1,\ldots,G_M$ of $T$ identical columns each. Thus 
$$G_i \,=\, \{(i-1)T+1,\ldots,iT\} \quad \mbox{for } i=1,\ldots,M.$$
In other words, the $M$ groups correspond to the $M$ points of the $(n-t)$-dimensional affine geometry $\Sigma = AG(n-t,3)$. 

\vspace{1mm}
We can now prove the following ternary analogue of Lemma \ref{blocksbinary}:

\begin{lem}\label{blocksternary}
Let $x$ and $y$ be two distinct points of $\cD$, and let $B=\{x,y,z\}$ be any block containing these two points. Then one of the following two cases occurs:
\begin{enumerate}
\item  If $x$ and $y$ belong to the same group, say $G_i$, then $z$ also belongs to $G_i$.
\item  If $x$ and $y$ belong to different groups $G_i$ and $G_j$, and if $\bar x$ and $\bar y$ denote the  (distinct) points of $\Sigma$ corresponding to $x$ and $y$, 
	then $\bar z$ is the third point on the line of $\Sigma$ through $\bar x$ and $\bar y$. 
	In other words, $z$ belongs to the group $G_k \neq G_i,G_j$ determined by the third point of the line $\bar x \bar y$ of $\Sigma$.
\end{enumerate}
In particular, any block $B$ of $\cD$ joining two points in distinct groups induces a line of $\Sigma$.
\end{lem}

\begin{proof}
Note first that it suffices to verify the statement in case (ii).
Thus let $B$ be a block containing two points $x$ and $y$ which belong to different groups $G_i$ and $G_j$, respectively.
Then $G_i$ and $G_j$  determine two distinct points $\bar x$ and $\bar y$ of $\Sigma$.
Let $\bar w$ be the third point on the line of $\Sigma$ through $\bar x$ and $\bar y$, and denote the group determined by this point by $G_k$ (so that $G_k \neq G_i,G_j$).

Let $z$ be the third point in $B$ and suppose $z \notin G_k$, say $z \in G_h$ with $h\neq k$ ($h=i$ or $h=j$ is still permissible at this point).
Clearly, $B' = \{x,y,w\}$ is a block, for every choice of a point $w \in G_k$. 
Then the difference $B-B'$ of the code words $B$ and $B'$ is a word $c$ of weight 2 in $C$, and the two non-zero entries of $c$ belong to the columns $w$ and $z$ of $H$,
which are in the distinct groups $G_k$ and $G_h$. 

Consider  the matrix obtained by using just one of the $T$ columns in each group, that is, 
the parity check matrix $H' := H_{n-t}$ for the ternary code $C'$ determined by the lines in the  affine geometry $\Sigma = AG(n-t,3)$.
As $c$ is orthogonal to all rows of $H$, the row vector $c'$ of length $3^{n-t}$ and weight 2 with  entry 1 in position $h$ and entry $2$ in position $k$ is orthogonal to all rows of $H'$.
This contradicts the well-known fact that the minimum weight vectors of $C'$ have weight 3 (they are the incidence vectors of  the lines of $\Sigma$).
Hence $z \in G_k$, as claimed. 
\end{proof}

We note that a weaker version of the special case $t=1$ of  Lemma \ref{blocksternary} (without the specific description of the third group $G_k$ in case (ii)) was already obtained in \cite[Theorem 2.2]{JMTW},
where also the special case $t=1$ of the following ternary analogue of Theorem \ref{GDDternary} was given.

\begin{thm}\label{GDDternary}
The triple system $\cD$ is a group divisible design with the $M$ groups  $G_1,\ldots,G_M$ of size $T$ each,
	 where two points in the same group are joined by $T-2$ blocks and two points in different groups are joined by $T$ blocks.
\end{thm}

\begin{proof}
This follows easily from Lemma \ref{blocksternary}: it suffices to observe that any choice of three points $x,y,z$ satisfying the conditions given in the lemma indeed yields a block.
\end{proof}

As in the binary case, we can also use Lemma \ref{blocksternary} to describe the structure of any STS$(N)$ contained in $\cD$:

\begin{thm}\label{STSternary}
Let $\cS$ be an arbitrary Steiner triple system STS($N$) contained in $\cD$. Then the block set $\cB$ of $\cS$ splits as follows:
\begin{itemize}
\item   for all $i=1,\ldots,M$, a set $\cB_i$ of $T(T-1)/6$ blocks such that $(G_i,\cB_i)$ is a Steiner triple system $\cS_i$ on the $T$ points in $G_i$;
\item   for each line $\ell$ of the affine geometry $\Sigma = AG(n-t,3)$, a set $\cB_\ell$ of $T^2 = 3^{2t}$ blocks forming a transversal design TD$[3;T]$ on the three groups determined by the points of $\ell$.
\end{itemize}
\end{thm}

\begin{proof}
By case (i) in Lemma \ref{blocksternary}, we necessarily obtain a sub-STS($T$) of $\cS$ on each group.
Then case (ii) implies that the remaining blocks of $\cS$ have to consist of points in three distinct groups which induce a line of $\Sigma$.
Now let $\ell = \{\bar x, \bar y, \bar z\}$ be such a line, and let the three groups determined by the points of $\ell$ be $G_i$, $G_j$ and $G_k$.
As any two points in different groups have to be on a unique block in $\cB$, the set of blocks of $\cS$ inducing the line $\ell$ obviously has to form a TD on these three groups. 
\end{proof}

As in the binary case, we can now establish a generic formula for the number of distinct Steiner triple systems contained in $\cD$.
In this formula, the functions $N_1(v)$ and $N_3(g)$ have the same meaning as in Section \ref{secbin} ($N_2(2k)$ is not needed in the ternary case).

\begin{thm}\label{nrSTSternary}
The number $s'(n,t)$ of distinct Steiner triple systems (with $3$-rank at most $3^n-n-1+t$) contained in the triple system $\cD$ 
formed by the supports of the words of weight $3$ (with non-zero entries $1$) in the ternary code $C$ with parity check matrix $H_{n,t}$, where $1 \leq t \leq n-1$, is given by
\begin{equation}\label{eqnrtern}
s'(n,t) \,=\, N_1(T)^M \cdot  N_3(T)^{M(M-1)/6},
\end{equation}
where  $T = 3^t$ and $M = 3^{n-t}$. 
\end{thm}

\begin{proof}
By Theorem \ref{STSternary}, we have to select an STS($T)$ on $G_i$ contained in $\cD$, for all $i=1,\ldots,M$. 
As the blocks of $\cD$ contained in $G_i$ form the complete design on $G_i$, we can always do this in $N_1(T)$  ways, which gives the first factor in the formula \eqref{eqnrtern}.

Also, given any line $\ell$ of $\Sigma$, we have to select a transversal design on the three groups determined by $\ell$.
Again, all triples consisting of one point in each of these groups are blocks of $\cD$, so that the required TD can be chosen in $N_3(T)$ ways.
This  has to be done for all $M(M-1)/6$ lines of $\Sigma$, which results in the second factor in formula \eqref{eqnrtern}.
\end{proof}

The only case of Theorem \ref{nrSTSternary} established previously is the recent evaluation $$s'(3,1)= 8,916,100,448,256$$ given in \cite{JMTW}.
In this paper, we had not yet found a theoretical argument and had to rely on a  computer evaluation, which was a rather non-trivial task.

We now give explicit formulas for the cases $t=1$ and $t=2$. In particular, we obtain the special case just discussed in a theoretical way (and with a nicer form of the resulting number).

\begin{cor}\label{nrSTSternary+1}
The number of distinct Steiner triple systems contained in the ternary code with parity check matrix $H_{n,1}$ is given by
\begin{equation}\label{eqnrtern1}
s'(n,1) \,=\, 12^{3^{n-2}(3^{n-1}-1)/2}
\end{equation}
\end{cor}

\begin{proof}
Here $T=3$ and $M=3^{n-1}$. Trivially, $N_1(3) = 1$. Also, it is easy  to check $N_3(3) = 12$ directly (or see  \cite[A002860]{OE}).  
\end{proof}

\begin{cor}\label{nrSTSternary+2}
The number of distinct Steiner triple systems contained in the code  with parity check matrix $H_{n,2}$ is given by
\begin{equation}\label{eqnrtern2}
s(n,2) \,=\, 840^{3^{n-2}} \cdot  5524751496156892842531225600^{3^{n-3}(3^{n-2}-1)/2}
\end{equation}
\end{cor}

\begin{proof}
Here $T=9$ and $M=3^{n-2}$. It is easy to see that $N_1(9) = 840$, since the affine plane $AG(3,2)$ is, up to isomorphism, the only STS(9).
The automorphism group of this plane is the group $AGL(2,3)$ of order 432, so that indeed $N_1(9) = 9!/432 = 840$.
The value  $N_3(9) = 5524751496156892842531225600$ is taken from \cite[A002860]{OE}.
\end{proof}

The next case, $t=3$, cannot be evaluated explicitly, as neither of the values $N_1(27)$ and $N_3(27)$ is known.

Consequences for the number of isomorphism classes of  STS$(3^n)$ with a given 3-rank  will be discussed in Section \ref{secnoniso}.

\section{The number of isomorphism classes of STS with classical parameters and small rank}\label{secnoniso}

In this section, we use the preceding enumeration results for the number of \emph{distinct} Steiner triple systems with classical parameters and prescribed rank
to obtain estimates for the number of \emph{isomorphism classes} of such triple systems.
This approach relies, as usual, on ``mass formulas'' derived via the automorphism group of the ambient code $C$ containing representatives of all STS in question;
for our context, it was first used  in \cite{T01} to obtain a lower bound on the number of isomorphism classes of STS($2^n-1)$ with 2-rank $2^n-n$.

The generic approach is as follows. Let $C$ be a code (for us, a binary or ternary code) and suppose that the code words of a given weight $w$ (for us, $w=3$) 
contain representatives for each isomorphism class of a type of design with block size $w$ (for us, STS$(2^n-1)$ or STS$(3^n)$ with restrictions on the 2-rank or 3-rank, respectively),
and assume that the number $s$ of distinct designs of this type contained in $C$ is known.

Now let $x$ denote the number of isomorphism classes of designs of the type in question supported by weight $w$ words in $C$ -- which we wish to determine or estimate --
and let $\cS_1, \ldots, \cS_x$ be a set of representatives (contained in $C$) for these isomorphism classes. Clearly, we have the equation (mass formula)
\begin{equation}\label{gencount1}
s \,=\, \sum_{i=1}^x \frac{|\aut C|}{|\aut \cS_i \cap \aut C|}, 
\end{equation}
where $\aut C$ and $\aut \cS_i$ denote the automorphism groups of $C$ and  $\cS_i$, respectively. 

Assume in addition that we know a common lower bound $u$ and a common upper bound $U$ for all the orders $|\aut \cS_i \cap \aut C|$, $i=1,\ldots,x$.
Then Equation \eqref{gencount1} implies the following estimate for the desired number $x$: 
\begin{equation}\label{gencount2}
u \cdot \frac{s}{|\aut C|} \,\leq\, x \,\leq\,  U \cdot \frac{s}{|\aut C|}.
\end{equation}

We now apply these observations to our situation, beginning with the ternary case (where the resulting formulas are a little simpler).
As the vast majority of Steiner triple systems are known to be \emph{rigid} (that is, they admit no non-trivial automorphisms), by a result of Babai \cite{B}, 
we have to use the trivial lower bound $u = 1$. While we could also use the trivial upper bound $U = |\aut C|$ and still get (asymptotically) interesting results,
we can apply a simple argument to give a much stronger upper bound:

\begin{lem}\label{upperternary}
Let $\cS$ be any Steiner triple system (with $3$-rank at most $N-n+t$) contained in the triple system $\cD$ 
formed by the supports of the words of weight $3$ (with non-zero entries $1$) in the ternary code $C$ with parity check matrix $H_{n,t}$ as in \eqref{hi}. Then 
\begin{equation}\label{equpperternary}
|\aut \cS|  \leq (T!)^{n-t+1} \cdot |AGL(n-t,3)|,
\end{equation}
where $T=3^t$.
\end{lem}

\begin{proof}
In view of the structure of $G = \aut C = \aut \cD$ and of $\cS$ as described in Theorems \ref{autgpternary} and \ref{STSternary},
any automorphism of $\cS$ has to induce a collineation of the affine geometry $\Sigma = AG(n-t,3)$ induced by the groups of the GDD  $\cD$ (see also Lemma \ref{blocksternary}).
Clearly, this action of $G$ on $\Sigma$ can at most give all of $\aut \Sigma = AGL(n-t,3)$, and it only remains to estimate the size of the kernel of the action.

Thus we have to consider those automorphisms $\alpha$ of  $\cS$ which fix every group of $\cD$.
Note that  $\alpha$ also has to induce an automorphism of each of the $M(M-1)/6$ (where again $M=3^{n-t}$) transversal designs TD[$3;T$] associated with the lines $\ell$ of $\Sigma$.
Given the action of $\alpha$ on  two groups of such a TD, the action on the third group is obviously uniquely determined.
Now select $n-t+1$ points of $\Sigma$ in general position.
It is clear that the subspace generated by such a set of points (via forming the closure under line taking) is all of $\Sigma$, 
so that $\alpha$ is uniquely determined by its action on the corresponding $n-t+1$ groups of $\cD$.
Trivially, we can have at most $T!$ different actions of automorphisms of this type on any given group, which results in the estimate in \eqref{equpperternary}.
\end{proof}

The preceding argument still gives a rather crude estimate, as we have not made any attempt to take the size of the automorphism groups of transversal designs  TD[$3;T$] into account.
Nevertheless, in view of Theorem \ref{autgpternary}, already the bound \eqref{equpperternary} beats the trivial upper bound by a huge factor, namely $(T!)^{M-n-t+1}$, where $M=3^{n-t}$.

\vspace{1mm}
We now plug the values obtained in Theorem \ref{autgpternary}, Theorem \ref{nrSTSternary} and Lemma \ref{upperternary}
into the bound \eqref{gencount2} and obtain the following general estimate for the ternary case:

\begin{thm}\label{isoSTSternary}
The number $nr'(n,t)$ of isomorphism classes of Steiner triple systems on $3^n$ points with $3$-rank at most $3^n-n-1+t$, where $1 \leq t \leq n-1$, satisfies
\begin{equation}\label{eqisotern}
\frac{N_1(T)^M \cdot  N_3(T)^{M(M-1)/6}}{\big(T!\big)^M \cdot |AGL(n-t,3)|}  \,\leq\, nr'(n,t)
	 \,\leq\,  \frac{N_1(T)^M \cdot  N_3(T)^{M(M-1)/6}}{\big(T!\big)^{M-n+t-1}},
\end{equation}
where  $T = 3^t$ and $M = 3^{n-t}$. 
\end{thm}

\begin{proof}
By Theorem \ref{allSTSternary}, every Steiner triple system with $3$-rank at most $3^n-n-1+t$ is contained (up to isomorphism) in the triple system $\cD$ 
formed by the supports of the words of weight $3$ (with non-zero entries $1$) in the ternary code $C$ with parity check matrix $H_{n,t}$ as in \eqref{hi}.  
\end{proof}

Let us state the special case $t=1$ of the preceding estimate explicitly, see Corollary \ref{nrSTSternary+1}:

\begin{cor}\label{isoSTSternary+1}
The number of isomorphism classes of Steiner triple systems  on $3^n$ points with $3$-rank at most $3^n-n$ satisfies
$$ \frac{12^{3^{n-2}(3^{n-1}-1)/2}}{6^{3^{n-1}} \cdot |AGL(n-1,3)|}  \,\leq\, nr'(n,1)
	 \,\leq\,  \frac{12^{3^{n-2}(3^{n-1}-1)/2}}{6^{3^{n-1}-n}}.\qed$$
\end{cor}

We leave it to the reader to write down the corresponding result for $t=2$, using the data given in (the proof of) Corollary \ref{nrSTSternary+2}.

\vspace{1mm}
It is perhaps even more interesting to ask for a lower bound on the number of isomorphism classes of Steiner triple systems  on $3^n$ points with $3$-rank \emph{exactly} $3^n-n-1+t$.
Here we obtain the following general result:

\begin{thm}\label{isoSTSternary=}
The number $nr'_=(n,t)$ of isomorphism classes of Steiner triple systems on $3^n$ points with $3$-rank exactly $3^n-n-1+t$, where $2 \leq t \leq n-1$, satisfies
$$nr'_=(n,t) \,\geq\, \frac{N_1(T)^M \cdot  N_3(T)^{M(M-1)/6}}{\big(T!\big)^M \cdot |AGL(n-t,3)|} -  \frac{N_1(T')^{M'} \cdot  N_3(T')^{M'(M'-1)/6}}{\big((T')!\big)^{M'-n+t-2}},$$
where  $T = 3^t$, $T' = 3^{t-1}$, $M = 3^{n-t}$ and $M' = 3^{n-t+1}$. 
\end{thm}

\begin{proof}
In view of Theorem \ref{allSTSternary}, we obtain the desired lower bound by subtracting the upper bound for the case of 3-rank at most $3^n-n-t-2$ in Theorem \ref{isoSTSternary}
from the lower bound for the case of 3-rank at most $3^n-n-t-1$ given there.
\end{proof}

Note that the case $t=1$ has to be treated separately, since Theorem \ref{isoSTSternary} does not apply to STS of minimal rank $3^n-n-1$.
Recalling that  $AG_1(n,3)$  is, up to isomorphism,  the unique STS$(3^n)$ with 3-rank $3^n-n-1$ (see Theorem \ref{ternaryHamada}), 
Corollary \ref{isoSTSternary+1} immediately gives the following bound.

\begin{thm}\label{isoSTSternary+1=}
The number of isomorphism classes of Steiner triple systems  on $3^n$ points with $3$-rank exactly $3^n-n$ satisfies
\begin{equation}\label{eqisoternt=1}
 nr'_=(n,1) \,\geq\,  \frac{12^{3^{n-2}(3^{n-1}-1)/2}}{6^{3^{n-1}} \cdot |AGL(n-1,3)|} - 1.
 \end{equation}
\end{thm}

These estimates can be marginally improved provided one knows the  number of \emph{distinct} STS with an exact given smaller 3-rank.
We will illustrate this idea for the case $t=1$, that is, 3-rank $3^n-n$. Here we obtain the following results, which are of some interest in their own right:

\begin{lem}\label{nrclasst=1}
The number of distinct Steiner triple systems with $3$-rank exactly $3^n-n-1$ (and hence isomorphic to $AG_1(n,3)$) contained in the ternary code with parity check matrix $H_{n,1}$ is given by
\begin{equation}\label{eqnrtern1=}
cl'(n,1) \,=\, \frac{6^{3^{n-1}}}{2 \cdot 3^n}.
\end{equation}
\end{lem}

\begin{proof}
We obtain a parity check matrix for the code $C'$ of a classical STS on $3^n$ points from the code $C$ with parity check matrix $H=H_{n,1}$ by adding a further row to $H$ 
so that, for each of the $3^{n-1}$ groups of $C$, all three possible entries $0,1,2$ appear. 
Any specific choice gives a code $C'$ containing a unique copy $\Sigma_{C'}$ of $AG_1(n,3)$, supported by the vectors of weight 3 (with entries 1) in $C'$.
Clearly, there are $6^{3^{n-1}}$ choices for the extra row, corresponding to the subgroup  of $G=\aut C$ isomorphic to the direct product of $3^{n-1}$ copies of $S_3$; see Theorem \ref{autgpternary}.
Hence any two classical STS contained in $C$ are equivalent under $\aut C$.

Now we have to determine how many distinct codes $C'$ we obtain  in this way.
In view of the preceding observations, this number is simply the size of the orbit of  a specific choice for $\Sigma_{C'}$ under $G$.
Thus we need to determine the (size of) the stabilizer of $\Sigma_{C'}$ in $G$, that is, the subgroup $S$ of $\aut \Sigma_{C'}$ fixing $C$.
Obviously, $S$ is just the group of collineations of $\Sigma_{C'}$  fixing a specified parallel class of lines.
(Note that the $3^{n-1}$ groups of $C$ indeed give a parallel class of lines in the affine geometry $\Sigma_{C'}$.)
Thus
$$|S| \,=\, \frac{2 \cdot |AGL(n,3)|}{3^n-1} \,=\, 2  \cdot 3^{n(n+1)/2}(3^{n-1}-1)(3^{n-2}-1) \cdots (3-1).$$
It follows that the number of classical STS contained in $C$  is
\begin{equation*}\begin{split}
 cl'(n,1) &\,=\, \frac{|G|}{|S|}  \,=\, \frac{6^{3^{n-1}}  \cdot |AGL(n-1,3)|}{|S|} \\[1mm]
 	&\,=\, \frac{6^{3^{n-1}}  \cdot 3^{n(n-1)/2}(3^{n-1}-1)(3^{n-2}-1) \cdots (3-1)}{2 \cdot 3^{n(n+1)/2}(3^{n-1}-1)(3^{n-2}-1) \cdots (3-1)}
	\,=\, \frac{6^{3^{n-1}}}{2 \cdot 3^n},
\end{split}\end{equation*}
as claimed.
\end{proof}

Subtracting the number in Equation \eqref{eqnrtern1=} from that in \eqref{eqnrtern1}, we get the following result:

\begin{thm}\label{nrSTSternary+1=}
The number of distinct Steiner triple systems of  $3$-rank exactly $3^n-n$ contained in the ternary code with parity check matrix $H_{n,1}$ is given by
\begin{equation}\label{eqnrtern1=}
s_='(n,1) \,=\, 12^{3^{n-2}(3^{n-1}-1)/2} -  \frac{6^{3^{n-1}}}{2 \cdot 3^n}
\end{equation}
\end{thm}

We can now apply the general bound \eqref{gencount2} by  taking $s$ as the number given in \eqref{eqnrtern1=} and plugging in the value for $|\aut C|$,
since $C$ obviously acts on the set of all Steiner triple systems of  $3$-rank exactly $3^n-n$ contained in $C$.
Unfortunately, this does not lead to a significant improvement of the bound in Theorem \ref{isoSTSternary+1=}, since we only get the minor strengthening
\begin{equation}\label{eqisoternt=1a}
 nr'_=(n,1) \,\geq\,  \frac{12^{3^{n-2}(3^{n-1}-1)/2}}{6^{3^{n-1}} \cdot |AGL(n-1,3)|} - \frac{1}{2 \cdot 3^n \cdot |AGL(n-1,3)|}.
 \end{equation}
As $ nr'_=(n,1)$ is an integer, this raises the lower bound provided by \eqref{eqisoternt=1} at most by  1.

Therefore, we feel that determining the precise number of distinct STS in the codes with parity check matrix $H=H_{n,t}$, $t \geq 2$, with an exact given  3-rank 
is not all that important if one wants to estimate the number of isomorphism classes of STS$(3^n)$ with  a prescribed 3-rank: 
the relatively easy Theorem \ref{isoSTSternary=} will do.

Nevertheless, we mention that it should be possible to count the  numbers of distinct Steiner triple systems of  $3$-ranks exactly 
$3^n-n-1$, $3^n-n$, and $3^n-n+1$ contained in the ternary code with parity check matrix $H=H_{n,2}$.
For this, one would need to analyze what happens if one adds one or two rows to $H$ (similar to the approach in the proof of Lemma \ref{nrclasst=1}), to go down with the rank by 1 or 2, respectively.
As this looks rather involved, we do not think that it is worth pursuing now.

\vspace{1mm}
Let us illustrate the results we have obtained for the ternary case by considering what they imply for the smallest interesting special case, that is, for STS(27),
which we have already investigated in \cite{JMTW}.

\begin{ex}\label{ex27}\rm
Putting $n=3$ in Corollary \ref{nrSTSternary+1}  shows that the number of distinct STS(27) in the ternary code with parity check matrix $H_{3,1}$ is 
$$12^{12} = 8,916,100,448,256$$
and that exactly
$$\frac{6^9}{2\cdot 3^3} \,=\, 2^8 \cdot 3^6 \,=\, 186,624$$
of these STS are classical, by  Lemma \ref{nrclasst=1}, whereas $8,916,100,261,632$ have 3-rank 24; this agrees with the values found in \cite{JMTW} via computer work.
Now Corollary \ref{isoSTSternary+1}  gives
\begin{equation}\label{STS27}
2048 \,=\, \frac{12^{12}}{6^{9} \cdot 432}  \,\leq\, nr'(3,1)  \,\leq\,  \frac{12^{12}}{6^6} \,=\, 191,102,976, 
\end{equation}
so that the number of isomorphism classes of STS(27) with 3-rank 24 is at least 2047, by Theorem \ref{isoSTSternary+1=}.
The more elaborate bound \eqref{eqisoternt=1a} improves the latter estimate by  1, giving at least 2048 isomorphism types of STS(27) with 3-rank 24,
as already computed in \cite{JMTW} using the same approach. 

After submitting the first version of the present paper, we managed to complete the classification of the STS(27) with 3-rank 24, 
using the general structural results obtained in Section \ref{sectern}: there are precisely 2624 isomorphism types, of which just one admits a point-transitive group; see \cite{JMTW2}.

We now substitute the values provided in Corollary \ref{nrSTSternary+2}  into the general estimate given in Theorem \ref{isoSTSternary} 
to obtain the following lower bound for the  number of isomorphism classes of STS(27) with 3-rank at most 25:
\begin{equation*}\begin{split}
nr'(3,2) &\,\geq\, \frac{s(3,2)}{(9!)^3 \cdot |AGL(1,3)|} \\[1mm]
  	&\,=\, \frac{840^3 \cdot  5524751496156892842531225600}{(9!)^3 \cdot 6} \\[1mm]
	& \,=\,
\frac{102790449873603788800}{9} \,>\, 1.14 \cdot 10^{19}.
\end{split}\end{equation*}
In view of the upper bound in \eqref{STS27}, we see that there are certainly more than $10^{19}$ isomorphism classes of STS(27) with 3-rank 25.
We also note that the upper bound provided by Theorem \ref{isoSTSternary}  exceeds the lower bound by a factor of
$$6 \cdot (9!)^2 \,=\, 790 091366 400,$$
and hence there are less than $10^{31}$ isomorphism classes of such STS.

Note that our bound  $nr'(3,2) > 1.14 \cdot 10^{19}$ for the number of STS(27) with 3-rank at most 25 is considerably larger than the bound $10^{11}$  
given in the CRC handbook \cite{CRC} for the total number of non-isomorphic STS(27) (without restriction on their 3-rank). \qed
\end{ex}

As is to be expected, we have completely analogous results for the binary case. 
Therefore, we will merely state these results and leave all details to the reader.

The approach used in the proof of Lemma \ref{upperternary} carries over to give the following binary analogue:

\begin{lem}\label{upperbinary}
Let $\cS$ be any Steiner triple system (with $2$-rank at most $N-n+t$) contained in the triple system $\cD$ 
formed by the supports of the words of weight $3$  in the binary code $C$ with parity check matrix $H_{n,t}$ (as in Section \ref{secbin}). Then 
\begin{equation}\label{equpperbinary}
|\aut \cS|  \leq  T! \cdot \big((T+1)!\big)^{n-t+1} \cdot |PGL(n-t,2)|,
\end{equation}
where $T=2^t-1$.  \qed
\end{lem}

Plugging the values obtained in Theorem \ref{autgpbinary}, Theorem \ref{nrSTSbinary} and Lemma \ref{upperbinary}
into the generic bound \eqref{gencount2} then yields the following general estimate for the binary case:

\begin{thm}\label{isoSTSternary}
The number $nr(n,t)$ of isomorphism classes of Steiner triple systems on $2^n-1$ points with $2$-rank at most $2^n-n-1+t$, where $1 \leq t \leq n-1$, satisfies
\begin{equation*}\begin{split}
&\frac{N_1(T) \cdot \big(N_2(T+1)\cdot T! \big)^M \cdot  N_3(T+1)^{M(M-1)/6}}{T! \cdot \big((T+1)!\big)^M \cdot |PGL(n-t,2)|} \\[1mm]
	 & \,\leq\, nr(n,t) \,\leq\,  \frac{N_1(T) \cdot \big(N_2(T+1)\cdot T! \big)^M \cdot  N_3(T+1)^{M(M-1)/6}}{\big((T+1)!\big)^{M-n+t-1}},
\end{split}\end{equation*}
where $T = 2^t-1$ and $M = 2^{n-t}-1$. 
\end{thm}

Again, we state the special case $t=1$ of the preceding estimate explicitly, see Corollary \ref{nrSTSbinary+1}:

\begin{cor}\label{isoSTStbinary+1}
The number of isomorphism classes of Steiner triple systems  on $2^n-1$ points with $2$-rank at most $2^n-n$ satisfies
$$ \frac{2^{(2^{n-1}-1)(2^{n-2}-1)/3}}{2^{2^{n-1}-1} \cdot |PGL(n-1,2)|}  \,\leq\, nr(n,1)
	 \,\leq\,  \frac{2^{(2^{n-1}-1)(2^{n-2}-1)/3}}{2^{2^{n-1}-n}}.$$
\end{cor}

We leave it to the reader to write down the corresponding results for $t=2$ and $t=3$, using the data given in (the proof of) Corollaries \ref{nrSTSbinary+2} and \ref{nrSTSbinary+3}.
For exact 2-ranks, we have the following analogue of Theorem \ref{isoSTSternary=}:

\begin{thm}\label{isoSTSbinary=}
The number $nr_=(n,t)$ of isomorphism classes of Steiner triple systems on $2^n-1$ points with $2$-rank exactly $2^n-n-1+t$, where $2 \leq t \leq n-1$, satisfies
\begin{equation*}\begin{split}
nr_=(n,t) \,\geq\, &\frac{N_1(T) \cdot \big(N_2(T+1)\cdot T! \big)^M \cdot  N_3(T+1)^{M(M-1)/6}}{T! \cdot \big((T+1)!\big)^M \cdot |PGL(n-t,2)|} \\[1mm]
	 &-  \frac{N_1(T') \cdot \big(N_2(T'+1)\cdot (T')! \big)^{M'} \cdot  N_3(T'+1)^{M'(M'-1)/6}}{\big((T'+1)!\big)^{M'-n+t-2}},
\end{split}\end{equation*}
where  $T = 2^t-1$, $T' = 2^{t-1}-1$, $M = 2^{n-t}-1$ and $M' = 2^{n-t+1}-1$. 
\end{thm}

For $t=1$, we recall that $PG_1(n-1,2)$  is, up to isomorphism,  the unique STS$(2^n-1)$ with 2-rank $2^n-n-1$, by Theorem \ref{binaryHamada}. This gives

\begin{thm}\label{isoSTSbinary+1=}
The number of isomorphism classes of Steiner triple systems  on $2^n-1$ points with $2$-rank exactly $2^n-n$ satisfies
\begin{equation}\label{eqisobint=1}
 nr_=(n,1) \,\geq\,   \frac{2^{(2^{n-1}-1)(2^{n-2}-1)/3}}{2^{2^{n-1}-1} \cdot |PGL(n-1,2)|}  - 1.
 \end{equation}
\end{thm}

We also note the following binary analogue of Lemma \ref{nrclasst=1} which was already established by Tonchev \cite{T01}: 

\begin{lem}\label{nrclasst=1bin}
The number of distinct Steiner triple systems with $2$-rank exactly $2^n-n-1$ (and hence isomorphic to $PG_1(n-1,2)$) contained in the binary code with parity check matrix $H_{n,1}$ is given by
\begin{equation}\label{eqnrbin1=}
cl(n,1) \,=\, 2^{2^{n-1}-n}.
\end{equation}
\end{lem}

Tonchev then stated a binary analogue of Theorem \ref{nrSTSternary+1=}, which we will not repeat here.
Using this result, the lower bound stated in Theorem \ref{isoSTSbinary+1=} can be improved by at most 1, as in the ternary case.

We also mention that the  numbers of distinct Steiner triple systems of  $2$-ranks exactly $2^n-n-1$, $2^n-n$, and $2^n-n+1$ contained in the ternary code 
with parity check matrix $H=H_{n,2}$ was computed by  Zinoviev and Zinoviev \cite{ZZ12,ZZ13a},
and a corresponding result for the case $t=3$ was given by Zinoviev \cite{Z16}. All these papers are quite involved.

We conclude also the binary case with an example:

\begin{ex}\label{ex31}\rm
Let us consider the first interesting case  $n=5$. (Note that the STS(15) and their codes have been classified, see \cite{TW}.) 
For $t=1$, Theorem \ref{isoSTSbinary+1=} shows 
$$nr(5,1) \,\geq\, \frac{2^{5\cdot 7}}{2^{15} \cdot (15 \cdot 14 \cdot 12 \cdot 8)} - 1\,=\, \frac{2^{14}}{315} -1  ·\,>\,  51.01,$$
so that there are at least  52 isomorphism classes of STS(31) with 2-rank 27. 
We remark in passing that the more elaborate bound based on Lemma \ref{nrclasst=1bin} does not give an improvement here (due to the necessary rounding to the next higher integer).
Also, Corollary \ref{isoSTStbinary+1} gives an upper bound of  $2^{25}$ for the case of 2-rank at most 27, so that our general estimates suffice to show
$$52 \,\leq\,  nr_=(5,1) \,\leq\, 33,554,431.$$
Actually,  the precise value of $nr_=(5,1)$ is known: the relevant mass formula has been used by Osuna \cite{O} to enumerate all  STS(31) with 2-rank 27.
This enumeration gave exactly 1239 isomorphism classes. 

Next, we apply Theorem \ref{isoSTSbinary=} with $t=2$.
Together with the data given in Corollary \ref{nrSTSbinary+2}, Theorem \ref{isoSTSternary} gives us the following lower bound 
for the number of isomorphism classes of STS(31) with 2-rank at most 28:
$$ \frac{6^7 \cdot 576^7}{6 \cdot (4!)^7 \cdot 168} \,=\, \frac{8,916,100,448,256}{7},$$
which implies $nr(5,2) \geq 1,273,728,635,466$. Subtracting our upper bound for the number of examples with 2-rank 26 or 27 (that is, $2^{25}$), we get
$$ nr_=(5,2) \,\geq\, 1,273,695,081,034 \approx 1.27 \cdot 10^{12}.$$
If we use  Osuna's precise evaluation of $nr_=(5,1)$ instead of our upper bound, we obtain only a  minor improvement, namely
$$ nr_=(5,2) \,\geq\, 1,273,728,634,227.$$
We note that Theorem \ref{isoSTSternary}  gives an upper bound of $$2,958,148,142,320,582,656 \approx 2.96 \cdot 10^{18}$$ 
for the number of isomorphism classes with rank at most 28.
Using Theorem \ref{isoSTSternary} together with the data in Corollary \ref{nrSTSbinary+3}, we get the following lower bound for 3-rank at most 29:
\begin{equation*}\begin{split}
 nr(5,2) \,&\geq\, \frac{30 \cdot 31449600^3 \cdot  108776032459082956800}{7! \cdot (8!)^3 \cdot 168} \\[1mm]
 	    \,&=\, 1,828,935,790,657,693,286,400,000 \approx 1.82 \cdot 10^{24},
\end{split}\end{equation*}
which results in an (only somewhat smaller) lower bound  for the number of isomorphism classes of STS(31) with 3-rank exactly 29:
$$nr_=(5,3) \,\geq\, 1,828,932,832,509,550,965,817,344.$$
As in the ternary case discussed in Example \ref{ex27}, these estimates are considerably larger than the bound $6 \cdot 10^{16}$ given in the CRC handbook \cite{CRC} for the total number of non-isomorphic STS(31).  \qed
\end{ex}

We finally note that it is also possible to obtain formulas for the total number of distinct STS($2^n-1)$ and STS($3^n$) with a prescribed (exact) 2- or 3-rank, respectively
(not just those contained in the relevant code $C$), provided that one knows the precise number $s$ of examples contained in $C$.
This has been done in the binary case for $t \leq 3$ in the papers \cite{T01,Z16,ZZ12,ZZ13a} cited above. 
As our main interest was in obtaining bounds on the number of non-isomorphic designs with a given rank and as finding the required exact numbers $s$ is rather involved,
we decided not to pursue this problem in the present paper.

\section{ Acknowledgements}

This work was initiated while the second author was visiting the University
of Augsburg as an Alexander von Humboldt Research Fellow. Vladimir Tonchev
thanks the University of Augsburg for the kind hospitality, and
acknowledges support by the Alexander von Humboldt Foundation
and  NSA Grant H98230-16-1-0011.
The authors wish to thank the unknown reviewers for reading carefully 
the manuscript 
and making several useful comments and suggestions.


\vspace{5mm}
\begin{thebibliography}{99}

\bibitem{A} E. F. Assmus, Jr.: On  $2$-ranks of Steiner triple systems.
{\it Electronic J. Combinatorics} {\bf 2} (1995), \#R9.

\bibitem{AK}  E. F. Assmus, Jr., J. D. Key: {\it Designs and their codes}. Cambridge University Press, Cambridge (1992).

\bibitem{B}  L. Babai: Almost all Steiner triple systems are asymmetric. 
{\it Ann. Discr. Math.} {\bf 7} (1980), 37--39.

\bibitem{BJL} T. Beth, D. Jungnickel and H. Lenz:
{\it Design Theory (2nd edition)}. Cambridge University Press, Cambridge (1999).

\bibitem{CRC}  C. J. Colbourn, J. F. Dinitz (Eds.): {\it Handbook of Combinatorial Designs (Second Edition)}. 
Chapman \& Hall/CRC, Boca Raton (2007).

\bibitem{DHV} J. Doyen, X. Hubaut, M. Vandensavel:
Ranks of incidence matrices of Steiner triple systems. {\it Math. Z.} {\bf 163} (1978), 251--259.

\bibitem{HP}  W. C. Huffman,  V. Pless: {\it Fundamentals of Error-Correcting Codes}.
Cambridge University Press, Cambridge (2003).

\bibitem{JMTW}  D. Jungnickel, S. S. Magliveras, V. D. Tonchev, A. Wassermann: 
On classifying Steiner triple systems by their 3-rank. 
{\it LNCS} {\bf 10693} (2018), 295 - 305.

\bibitem{JMTW2}  D. Jungnickel, S. S. Magliveras, V. D. Tonchev, A. Wassermann: 
The classification of Steiner triple systems on 27 points with  3-rank 24.
\emph{Des. Codes Cryptogr.}  (2018), DOI 10.1007/s10623-018-0502-5

\bibitem{JT}  D. Jungnickel, V. D. Tonchev:  On Bonisoli's theorem and the block 
codes of Steiner triple systems.
\emph{Des. Codes Cryptogr.} {\bf 86} (2018), 449-462.

\bibitem{OE}  \emph{The On-Line Encyclopedia of Integer Sequences}. https://oeis.org

\bibitem{O} O. P. Osuna: There are 1239 Steiner triple systems $STS(31)$ of 2-rank 27.
\emph{Des. Codes Cryptogr.} {\bf 40}  (2006),  187--190.

\bibitem{T01} V. D. Tonchev:  A mass formula for Steiner triple systems STS$(2^{n}-1)$ of 2-rank $2^n -n$.
\emph{ J. Combin. Th.} Ser. A {\bf 95} (2001), 197--208.

\bibitem{TW} V. D. Tonchev, R. S. Weishaar: Steiner triple systems of order 15 and their codes.
{\it J. Statist. Planning and Inference} {\bf 58} (1997), 207--216.

\bibitem{Z16}   D. V. Zinoviev: The number of Steiner triple systems $S(2^m - 1, 3, 2)$ of rank $2^m -m+2$ over $\mathbb{F}_2$.
\emph{Discr. Math.} {\bf 339} (2016), 2727--2736.

\bibitem{ZZ12}  V. A. Zinoviev, D. V. Zinoviev: Steiner triple systems $S(2^m - 1, 3, 2)$ of rank $2^m -m+1$ over $\mathbb{F}_2$.
\emph{Problems of Information Transmission} {\bf 48} (2012), 102--126.

\bibitem{ZZ13}  V. A. Zinoviev, D. V. Zinoviev: Structure of Steiner triple systems $S(2^m - 1, 3, 2)$ of rank $2^m -m+2$ over $\mathbb{F}_2$.
\emph{Problems of Information Transmission} {\bf 49} (2013), 232--248.

\bibitem{ZZ13a}  V. A. Zinoviev, D. V. Zinoviev: 
Remark on `Steiner triple systems $S(2^m - 1, 3, 2)$ of rank $2^m -m+1$ over $\mathbb{F}_2$ published in \emph{Probl. Peredachi Inf.}, 2012, no. 2.''
\emph{Problems of Information Transmission} {\bf 49} (2013), 107--111.


\end{thebibliography}
\end{document}